\documentclass[12pt,twoside,reqno,psamsfonts]{amsart}

\usepackage[OT1]{fontenc}
\usepackage{type1cm}
\usepackage{enumerate}
\usepackage{amssymb}
\usepackage[dvips]{graphicx}
\usepackage{geometry}        
\usepackage{version}         

\numberwithin{equation}{section}

\theoremstyle{plain}
\newtheorem{theorem}{Theorem}[section]

\theoremstyle{remark}
\newtheorem{remark}[theorem]{Remark}

\newtheorem*{ack}{Acknowledgement}

\theoremstyle{definition}

\newcommand{\QQ}{\mathcal{Q}}

\newcommand{\Z}{\mathbb{Z}}
\newcommand{\N}{\mathbb{N}}

\newcommand{\eps}{\varepsilon}

\DeclareMathOperator{\dist}{dist}

\DeclareMathOperator{\inter}{int}

\DeclareMathOperator{\spt}{spt}

\DeclareMathOperator{\udimloc}{\overline{dim}_{loc}}
\DeclareMathOperator{\ldimloc}{\underline{dim}_{loc}}

\begin{document}

\title[Existence of doubling measures]{Existence of doubling measures
  via generalised nested cubes}

\author{Antti K\"aenm\"aki}
\address{Department of Mathematics and Statistics \\
         P.O. Box 35 (MaD) \\
         FI-40014 University of Jyv\"askyl\"a \\
         Finland}
\email{antti.kaenmaki@jyu.fi}

\author{Tapio Rajala}
\address{Scuola Normale Superiore \\
         Piazza dei Cavalieri 7 \\
         I56127 Pisa \\
         Italy}
\email{tapio.rajala@sns.it}

\author{Ville Suomala}
\address{Department of Mathematics and Statistics \\
         P.O. Box 35 (MaD) \\
         FI-40014 University of Jyv\"askyl\"a \\
         Finland}
\email{ville.suomala@jyu.fi}

\thanks{V.\ Suomala acknowledges the support of the Academy of Finland, project \#126976}

\begin{abstract}
Working on doubling metric spaces, we construct generalised dyadic cubes adapting ultrametric structure. If the space is complete, then the existence of such cubes and the mass distribution principle lead into a simple proof for the existence of doubling measures. As an application, we show that for each $\eps>0$ there is a doubling measure having full measure on a set of packing dimension at most $\eps$.
\end{abstract}

\maketitle

\section{Introduction and notation}

A measure $\mu$ on a metric space $(X,d)$ is called \emph{doubling} if
there is a constant $1 \le D < \infty$ such that
\[
  0 < \mu\bigl( B(x,2r) \bigr) \leq D\mu\bigl( B(x,r) \bigr) < \infty
\]
for all $x \in X$ and $r>0$. Here $B(x,r) = \{ y \in X : d(x,y) \le r
\}$ is a closed ball with centre $x$ and radius $r$. We
denote open balls by $U(x,r)$. By a measure we always mean a Borel regular
outer measure. A metric space $X$ has the \emph{finite doubling property} if any
ball $B(x,2r)\subset X$ may be covered by finitely many balls of
radius $r$. Furthermore, such a space is \emph{doubling} if the number
of the $r$-balls needed to cover $B(x,2r)$ has an upper bound
$N\in\mathbb{N}$ independent of $x$ and $r$.

Let $\mathcal{D}(X)$ be the collection of all
doubling measures on $X$.
It is clear that if $\mathcal{D}(X)\ne\emptyset$, then $X$ is
doubling. The reverse implication is true if $X$ is assumed to be
complete.
For compact
doubling metric spaces this result was first proved by
Vol'berg and Konyagin \cite{VolbergKonyagin1984,VolbergKonyagin1987}. Luukkainen and
Saksman \cite{LuukkainenSaksman1998} generalised it to the complete
case. A slightly simpler proof in the compact case can be found in Wu
\cite{Wu1998} (see also Heinonen \cite{Heinonen2001}). Saksman
\cite{Saksman1999} has constructed examples of domains
$\Omega$ with $\mathcal{D}(\Omega)=\emptyset$
and the results of Cs\"ornyei and Suomala \cite{CsornyeiSuomala2010}
may be used to study whether $\mathcal{D}(X)\neq\emptyset$ for certain
countable sets $X\subset\mathbb{R}$.

In Theorem \ref{lemma:dyadic}, we construct nested families of ``cubes'' sharing most of the good properties of dyadic (or $r$-adic) cubes of Euclidean spaces. The existence of similar kind of nested partitions of metric spaces has been studied in many works (see e.g.\ \cite{Larman1967, Rogers1970, Christ1990, HytonenMartikainen2009}). Besides that our construction adapts ultrametric structure, we believe, given the length of the construction, it also fulfills the desire of the existence of a simple and effective construction of such cubes (see \cite[\S 7.2]{Rogers1970}).

As an application, we show in Theorem \ref{thm:main} that complete doubling metric spaces carry doubling measures. This is of course a known result (see \cite{VolbergKonyagin1987, Wu1998, LuukkainenSaksman1998}). Since our proof uses nested partitions and the mass distribution principle, we are able to get the result directly also in the unbounded case; the extra step in the unbounded case (see \cite{LuukkainenSaksman1998}) is not needed. Also, once the partitions have been fixed, our method can be used to construct doubling measures possessing certain measure theoretical self-similarity; see Remark \ref{rem}(3). In fact, in Theorem \ref{thm:wu} we show that for each $\eps>0$ there is a doubling measure having full measure on a set of packing dimension at most $\eps$. This result was known before only for the Hausdorff dimension (see \cite{Wu1998}).

In what follows, the (metric) closure, interior and boundary of a set
$A\subset X$ are denoted by
$\overline{A}$, $\inter(A)$ and $\partial A$, respectively.

\section{Construction of generalised dyadic cubes}

\begin{theorem}\label{lemma:dyadic}
  If $X$ is a metric space with the finite doubling property and $0<r<\frac13$, then there exists a collection $\{ Q_{k,i} : k \in \mathbb{Z},\; i \in N_k
  \subset \N \}$ of Borel sets having the following properties:
  \begin{enumerate}[(i)]
    \item\label{i} $X = \bigcup_{i\in N_k}Q_{k,i}$ for every $k \in \mathbb{Z}$,
    \item\label{ii} $Q_{k,i} \cap Q_{m,j} = \emptyset$ or $Q_{k,i}
      \subset Q_{m,j}$ when $k,m\in\mathbb{Z}$, $k\ge m$, $i \in N_k$ and $j \in N_m$,
    \item\label{iii} for every $k\in \mathbb{Z}$ and $i \in N_k$ there exists a point $x_{k,i} \in X$ so that
    \[
      U(x_{k,i},c r^k) \subset Q_{k,i} \subset B(x_{k,i},Cr^k)
    \]
where $c=\tfrac12-\tfrac{r}{1-r}$ and $C=\tfrac{1}{1-r}$,
     \item\label{iv} there exists a point $x_0 \in X$ so that for
       every $k \in \mathbb{Z}$ there is $i \in N_k$ so that
     \[
       U(x_0 ,cr^k) \subset Q_{k,i},
     \]
\item\label{v} $\{x_{k,i}\,:\,i\in N_k\}\subset\{x_{k+1,i}\,:\,i\in
  N_{k+1}\}$ for all $k\in\Z$.
  \end{enumerate}
\end{theorem}

\begin{proof}
  Fix a point $x_0 \in X$ and start by choosing a
  maximal collection of points $\{x_{0,i}\,:\,i\in N_0\}\subset X$
  containing $x_0$ and having the property that
  $d(x_{0,i},x_{0,j})\ge 1$ if $i\neq j$. Next, for each $k\in\mathbb{N}$, let
  $\{x_{k,i}\,:\,i\in N_k\}\supset\{x_{k-1,i}\,:\,i\in  N_{k-1}\}$ be
  a maximal collection of points having mutual distances at least
  $r^k$. If
  $k\in\mathbb{Z}$, $k<0$, we let
  $\{x_{k,i}\,:\,i\in N_{k}\}$ be any maximal subcollection of
  $\{x_{k+1,i}\,:\,i\in N_{k+1}\}$ containing $x_0$ whose points have mutual distances
  at least $r^{k-1}$.

  In the set of all possible pairs $(k,i)$, $k\in \mathbb{Z}$, $i\in
  N_k$, consider the smallest  partial order $\prec$ that satisfies the following
  property: For each $k \in \mathbb{Z}$ and $i \in N_{k+1}$, we have
  $(k+1,i) \prec (k,j)$ if
  \[
    j = \min\bigl\{ h \in N_k : \dist(x_{k+1,i},x_{k,h}) = \min_{l\in N_k}\dist(x_{k+1,i},x_{k,l}) \bigr\}.
  \]
  Notice that $\min_{l\in N_k}\dist(x_{k+1,i},x_{k,l})$ exists,
  because $X$ has the finite doubling property, and that $j$ also
  exists and is unique. The sets $Q_{k,i}$ will be defined by using
  this partial order.

  We first define the sets $Q_{0,i}$ for $i\in N_0$ as
  \[
    Q_{0,i} = \overline{\{x_{l,j}\,:\,(l,j)\prec(0,i)\}} \setminus \bigcup_{j<i}Q_{0,j}.
  \]
  For $k<0$ we define the sets $Q_{k,i}$ inductively as
  \[
    Q_{k,i}=\bigcup_{(k+1,j)\prec(k,i)}Q_{k+1,j}
  \]
  whereas for $k>0$, we let
  \[
    Q_{k,i} = Q_{k-1,j} \cap \overline{\{x_{l,j}\,:\,(l,j)\prec(k,i)\}} \setminus \bigcup_{j<i}Q_{k,j},
  \]
  where $(k,i)\prec(k-1,j)$. The defined sets are clearly Borel.

  Let us check that the sets $Q_{k,i}$ satisfy the conditions
  \eqref{i}--\eqref{v}. To see \eqref{i} it is enough to notice that
  \[
    \bigcup_{i\in N_k}Q_{k,i} = \bigcup_{i\in N_k} \overline{\{x_{l,j}\,:\,(l,j)\prec(k,i)\}},
  \]
  which is dense and closed in $X$. Conditions \eqref{ii} and \eqref{v} follow
  immediately from the construction. Let us next verify
  \eqref{iii}.
The fact that $\{x_{m,i}\,:\,i\in N_m\}$ is a maximal $r^m$-separated subset of $\{x_{m+1,j}\,:\, j\in N_{m+1}\}$ together with the definition of $\prec$ implies that $d(x_{m,i},x_{m+1,j})\leq r^{m}$ if $(m+1,j)\prec(m,i)$. This gives
  \[
    Q_{k,i} \subset B\bigl( x_{k,i}, \sum_{m=k}^\infty r^m \bigr) = B\left(x_{k,i},\tfrac{1}{1-r}r^k\right).
  \]
  On the other hand, if $n>k$ and $(n,j)\not\prec(k,i)$, then
  $d(x_{k,i},x_{n,j})\ge\tfrac12r^k-\sum_{m=k+1}^\infty
  r^m=(\tfrac12-\tfrac{r}{1-r})r^k$. This implies
\[U\left(x_{k,i},(\tfrac12-\tfrac{r}{1-r})r^k\right)
  \subset Q_{k,i}\]
 and finishes the proof of \eqref{iii}. Finally, the claim \eqref{iv} follows from \eqref{iii} since
 $x_0\in\{x_{k,i}\,:\,i\in N_k\}$ for all $k\in\mathbb{Z}$.
\end{proof}

\begin{remark}
The statement of Theorem \ref{lemma:dyadic} remains true also for any $\tfrac13\le r<1$ with $\tfrac12-\tfrac{r}{1-r}$ and $\tfrac{1}{1-r}$ replaced by some constants $0<c<C<\infty$ depending only on $r$. This can be easily seen as follows: First apply Theorem \ref{lemma:dyadic} with $r=\tfrac14$ to obtain the families $\{Q_{k,i} : i\in N_k\}$. Then, for any $\tfrac13\le\tilde{r}<1$ and $n\in\N$, we choose $k=k(n,\tilde{r})\in\N$ such that $4^{-k}<\tilde{r}^n\le 4^{-k+1}$. Now $\{\widetilde{Q}_{n,i}\}=\{Q_{k(n,\tilde{r}),i}\,:\,i\in N_{k(n,\tilde{r})}\}$ are the desired families. We formulated the result as in Theorem \ref{lemma:dyadic} since the explicit expressions for the constants $c$ and $C$ for small $r$ are needed in the proof of Theorem \ref{thm:main} below.
\end{remark}

\section{Existence of doubling measures}

\begin{theorem}\label{thm:main}
  If $X \ne \emptyset$ is a complete doubling metric space, then $\mathcal{D}(X) \ne \emptyset$.
\end{theorem}

\begin{proof}
  Fix $0<r\le\tfrac17$ and let $\QQ = \{ Q_{k,i} : k \in \Z,\; i \in N_k
  \}$, $\{ x_{k,i} : k \in \Z,\; i \in N_k \}$, and the constants $0<c<C<\infty$ be as in Theorem
  \ref{lemma:dyadic}. Let
\begin{equation}\label{eq:Mki}
M_{k,i} = \#\{ j \in N_{k+1} : Q_{k+1,j}
  \subset Q_{k,i} \}-1.
\end{equation}
Since $X$ is doubling, it follows using Theorem
  \ref{lemma:dyadic} \eqref{ii} and \eqref{iii} that there is $M\in\mathbb{N}$ such
  that $M_{k,i} \le M$ for every $k \in \Z$ and $i \in N_k$.
We define a code tree $\Sigma$ by setting $\Sigma = \{ (i_k)_{k
    \in \Z} : Q_{k,i_k} \subset Q_{k-1,i_{k-1}} \text{ for all } k \in
  \Z \}$ and equip this with the usual ultrametric: the distance between two
  different codes $(i_k)$ and $(j_k)$ is $2^{-n}$, where $n$ is the first index at which
  the codes differ, $n=\min\{k\,:\,i_k\neq j_k\}$. Theorem
  \ref{lemma:dyadic} \eqref{ii} and \eqref{iv} guarantee that this metric
  is well defined. We also define cylinders $[k,i] = \{ (j_n)_{n \in \Z}\in\Sigma : j_k = i
  \}$ and set $\Sigma_* = \{ [k,i] : k \in \Z \text{ and
  } i \in N_k \}$.
Since $X$ is complete, we may define a
  \emph{projection} $\pi \colon \Sigma \to X$ by the relation $\{
  \pi\bigl( (i_k)_{k \in \Z} \bigr) \} = \bigcap_{k \in \Z}
  \overline{Q_{k,i_k}}$. Now we clearly have $\pi([k,i]) =
  \overline{Q_{k,i}}$ for every $k \in \Z$ and $i \in N_k$.

  We define a set function $\nu \colon \Sigma_* \to [0,\infty)$ by
  first choosing $0<p<1/(M+1)$, $i_0\in N_0$ and setting $\nu([0,i_0]) =
  1$ and then requiring that for every $k \in \mathbb{Z}$ and $i \in
  N_k$ we have
  \begin{equation}\label{nudef}
    \nu([k+1,i]) =
\begin{cases}
       p\nu([k,j]), &\text{if } Q_{k+1,i} \subset Q_{k,j} \text{ and } x_{k+1,i} \ne x_{k,j}, \\
       (1-M_{k,i}p)\nu([k,j]), &\text{if } Q_{k+1,i} \subset Q_{k,j} \text{ and } x_{k+1,i} = x_{k,j}.
     \end{cases}
  \end{equation}
We may now easily extend $\nu$ to a measure on $\Sigma$ by setting
\[\nu(A)=\inf\Bigl\{\sum_{j}\nu([k_j,i_j])\,:\,A\subset \bigcup_{j}[k_j,i_j]\Bigr\}.\]
for all $A\subset X$. The main
reason for this to work is the fact that the cylinders $[k,j]$ are both open and closed
(compact) in $\Sigma$. See also \cite[\S
10]{Halmos1950}. It follows immediately from the construction that
$\nu$ is a doubling measure on $\Sigma$.

  Let $\mu = \pi\nu$ be the projected measure on $X$ given by
  $\mu(A)=\nu(\pi^{-1}(A))$ for all $A\subset X$. It is then clear that
  we have the estimates
\begin{equation}\label{reunaestimaatti}
\mu\bigl( \inter(Q_{k,i}) \bigr) \le \nu([k,i]) \le
  \mu(\overline{Q_{k,i}})
\end{equation}
for every $k \in \Z$ and $i \in N_k$.
Let
  us next show that this can be sharpened to
\begin{equation}\label{nollareuna}
\mu(Q_{k,i}) = \nu([k,i]).
\end{equation}
Fix $k \in \Z$, $i
  \in N_k$ and $(j_n)_{n \in \Z} \in \pi^{-1}(\partial
  Q_{k,i})$. Then $x := \pi\bigl( (j_n)_{n \in \Z} \bigr) \in
  \overline{Q_{n,j_n}}$ for every $n \in \Z$. Theorem
  \ref{lemma:dyadic}(iii) together with the fact $r\le1/7$ implies that
\[\overline{Q_{n+1,l_{n+1}}}\subset B(x_{n,j_n},Cr^{n+1})
  \subset U(x_{n,j_n},c r^n) \subset Q_{n,j_n}\]
 for every $n \in \Z$,
  where $l_{n+1} \in N_{n+1}$ satisfies $x_{n+1,l_{n+1}} =
  x_{n,j_n}$. Hence $Q_{n+1,l_{n+1}}\cap\partial Q_{k,i}=\emptyset$ and
    so $[n+1,l_{n+1}] \subset [n,j_n] \setminus
  \pi^{-1}(\partial Q_{k,i})$ for every  $n\in\Z$.
This means that $\pi^{-1}(\partial Q_{k,i})$ is a porous subset of
$\Sigma$ and since $\nu$ is doubling, it
  now follows that $\mu(\partial Q_{k,i})=\nu\bigl( \pi^{-1}(\partial Q_{k,i}) \bigr)
  = 0$ (See e.g. \cite[Proposition
  3.4]{{JarvenpaaJarvenpaaKaenmakiRajalaRogovinSuomala2007}} for a
  proof of this elementary fact). Combining this with
  \eqref{reunaestimaatti} implies \eqref{nollareuna}.

  It remains to show that $\mu$ is doubling. If $y \in X$ and $t>0$,
  let $k\in\mathbb{Z}$ be such that $3r^k \le t<3r^{k-1}$. By Theorem
  \ref{lemma:dyadic}(iii), the ball $B(y,t)$ contains $Q_{k,i}$ for
  some $i\in N_k$. On the other hand, the ball $B(y,2t)$ intersects
  $Q_{k,j}$ for at most $\tilde{M}$ indices $j\in N_k$, where
  $\tilde{M}<\infty$ depends only on $r$ and the doubling constant $N$
  of $X$. Thus it suffices to show that there exists a constant $1 \le \tilde{C} < \infty$ so that
  \begin{equation}\label{finalclaim}
    \mu(Q_{k,j}) \le \tilde{C}\mu(Q_{k,i})
  \end{equation}
  whenever $Q_{k,j}\cap B(y,2t)\neq\emptyset$. Fix $j \in N_k$ for
  which $Q_{k,j}\cap B(y,2t)\neq\emptyset$. We may assume that $i\neq
  j$ as otherwise \eqref{finalclaim} holds trivially. Observe that
  \begin{equation}\label{dest}
    d(x_{k,i},x_{k,j})\leq 3t+C r^k< r^{k-3}.
  \end{equation}
  Let $m$ be the largest integer such that $Q_{k,i} \cup Q_{k,j} \subset Q_{m,l}$ for some $l\in N_m$.
  For each $m\le n\le k$ let $i_{n},j_n \in N_n$ be the indices that
  satisfy $Q_{k,i} \subset Q_{n,i_n}$ and $Q_{k,j} \subset
  Q_{n,j_n}$.
  If $m<n\le k-4$, it follows that
  \begin{equation}\label{keskipakko}
    x_{n,j_n}\neq x_{n+1,j_{n+1}}\text{ and }
    x_{n,i_{n}}\neq x_{n+1,i_{n+1}}
  \end{equation}
  as otherwise Theorem \ref{lemma:dyadic} implies (recall $r\le\tfrac17$)
\[d(x_{k,i},x_{k,j})>cr^n-Cr^{n+1}=\left(\tfrac{1}{2}-\tfrac{2r}{1-r}\right)r^n\ge
r^{n+1}\ge r^{k-3}\]
contrary to \eqref{dest}.
Now \eqref{nollareuna}, \eqref{keskipakko}, and \eqref{nudef} imply
that $\mu(Q_{n,j_n}) = p\mu(Q_{n+1,j_{n+1}})$ and $\mu(Q_{n,i_n}) = p\mu(Q_{n+1,i_{n+1}})$ and thus
\[\frac{\mu(Q_{n+1,j_{n+1}})}{\mu(Q_{n,j_{n}})} \frac{\mu(Q_{n,i_{n}})}{\mu(Q_{n+1,i_{n+1}})}=1\]
for $m<n\le k-4$. Hence
  \begin{align*}
    \frac{\mu(Q_{k,j})}{\mu(Q_{k,i})} = \prod_{n=m}^{k-1} \frac{\mu(Q_{n+1,j_{n+1}})}{\mu(Q_{n,j_{n}})} \frac{\mu(Q_{n,i_{n}})}{\mu(Q_{n+1,i_{n+1}})} \le p^{-4}
  \end{align*}
  giving \eqref{finalclaim} and finishing the proof.
\end{proof}

\section{Dimension of doubling measures}\label{dim}

The local $L^q$-spectrum was recently introduced in \cite{KaenmakiRajalaSuomala2010}. It gives one way to quantify the local homogeneity of a given measure and can also be used to estimate the local dimensions.

Suppose that $\mu$ is a measure on a doubling metric space $X$ so that bounded sets have finite measure. Denote the support of $\mu$ by $\spt(\mu)$. If $0<r<1$, $x \in \spt(\mu)$, and $q \ge 0$, then, following \cite[Theorem 4.4]{KaenmakiRajalaSuomala2010}, we define the \emph{local $L^q$-spectrum of $\mu$ at $x$} by setting
\begin{equation*}
  \tau_q(\mu,x) = \lim_{t \downarrow 0} \liminf_{k \to \infty} \frac{\log\sum_{Q \in \mathcal{Q}_k(x,t)} \mu(Q)^q}{k\log r},
\end{equation*}
where $\mathcal{Q}_k(x,t) = \{Q_{k,i}\cap B(x,t) : i\in N_k\}$ and the collection $\{ Q_{k,i} : k \in \mathbb{Z},\; i \in N_k \}$ is as in Theorem \ref{lemma:dyadic}.

Let
\begin{equation*}
  \udimloc(\mu,x) = \limsup_{t \downarrow 0} \frac{\log\mu\bigl( B(x,t) \bigr)}{\log t}
\end{equation*}
be the \emph{upper local dimension of $\mu$ at $x$}. If instead of $\limsup_{t \downarrow 0}$ we take $\liminf_{t \downarrow 0}$, we get the \emph{lower local dimension of $\mu$ at $x$}, denoted by $\ldimloc(\mu,x)$ (see \cite[\S 10.1]{Falconer1997}). According to \cite[Theorem 4.2]{KaenmakiRajalaSuomala2010}, we have
\begin{equation} \label{eq:Lq}
  \lim_{q \downarrow 1} \frac{\tau_q(\mu,x)}{q-1} \le \ldimloc(\mu,x) \le \udimloc(\mu,x) \le \lim_{q \uparrow 1} \frac{\tau_q(\mu,x)}{q-1}
\end{equation}
for $\mu$-almost all $x \in X$. This estimate generalises the results \cite[Theorem 1.1]{Ngai1997}, \cite[Theorems 1.3 and 4.1]{Heurteaux1998}, and \cite[Theorem 1.4]{FanLauRao2002}.

We now apply \eqref{eq:Lq} for the measures constructed in the proof of Theorem \ref{thm:main} to estimate their packing dimension.

\begin{theorem} \label{thm:wu}
  If $X \ne \emptyset$ is a complete doubling metric space, then for every $\eps>0$ there is a doubling measure $\mu$ on $X$ such that
  \[
    \udimloc(\mu,x) \le \eps
  \]
  for $\mu$-almost all $x \in X$.
\end{theorem}

\begin{proof}
  Fix $0<r\le\tfrac17$ and let $\QQ = \{ Q_{k,i} : k \in \Z,\; i \in N_k
  \}$ be as in Theorem \ref{lemma:dyadic}. Let $\mu$ be the measure constructed in the proof of Theorem \ref{thm:main}.
Fix $k\in\mathbb{Z}$ and $i\in N_k$ and let $M_{k,i}$ be as in
\eqref{eq:Mki}. Since $M_{k,i}\le M$ and $p\le\tfrac{1}{M_{k,i}+1}$, we
  get $M_{k,i}\,p^q+(1-M_{k,i}\,p)^q\le Mp^q+(1-Mp)^q$ for all
  $0<q<1$. Recalling \eqref{nudef} and \eqref{finalclaim}, this implies
\begin{align*}
\sum_{Q_{k+1,j}\subset
Q_{k,i}}\mu(Q_{k+1,j})^q&=\mu(Q_{k,i})^q\bigl(M_{k,i}\,p^q+(1-M_{k,i}\,p)^q\bigr)\\
&\le\mu(Q_{k,i})\bigl(M p^q+(1-Mp)^q\bigr).
\end{align*}
Using this estimate recursively leads to $\tau_{q}(\mu,x)\ge \log\bigl(Mp^q+(1-Mp)^q\bigr)/\log r$ for all $x\in
X$ and $0<q<1$. Combining this with \eqref{eq:Lq} gives
\begin{equation}\label{dim}
\udimloc(\mu,x)\le\lim_{q\uparrow
  1}\frac{\tau_q(\mu,x)}{q-1}\le\frac{Mp \log p+(1-Mp)\log(1-Mp)}{\log
  r}
\end{equation}
for $\mu$-almost all $x\in X$.

As the upper bound in \eqref{dim} can be made arbitrarily small by
choosing $p>0$ small enough in \eqref{nudef}, we have shown the claim.
\end{proof}

\section{Further remarks}

\begin{remark} \label{rem}
(1) It is tempting to try to define the measure $\mu$ in the
proof of Theorem \ref{thm:main} directly without going into the code space $\Sigma$. More precisely, first
defining $\tilde{\mu}(Q_{k,i})$ as $\nu([k,i])$ in \eqref{nudef} and
then letting
$\mu(A)=\inf\{\sum_{j}\tilde{\mu}(Q_{k_j,i_j}) : A\subset\bigcup_j Q_{k_j,i_j}\}$
for $A\subset X$. Although it now follows from the proof of
Theorem \ref{thm:main}, it is not a priori clear that
$\mu(Q_{k,i})=\tilde{\mu}(Q_{k,i})$ for all $k$ and $i$. Observe
that in the code space $\Sigma$ this is not a problem since the
cylinders $[k,i]$ are both open and compact.

(2) The authors of the articles \cite{VolbergKonyagin1987,
    LuukkainenSaksman1998} prove not only the existence of
  doubling
  measures but also the existence of $\alpha$-homogeneous measures for each
  $\alpha$ strictly larger than the Assouad dimension of $X$ (see
  e.g. \cite{Heinonen2001} for the definitions). It is an
  easy exercise to check that given such $\alpha$, if $r>0$ is small
  enough and $p=r^{-\beta}$ in the proof of Theorem \ref{thm:main}
  above, where $\beta$ is between $\alpha$ and the Assouad dimension
  of $X$, then the measure $\mu$ will be $\alpha$-homogeneous.

(3) Using precisely the same idea as in the proof of Theorem
  \ref{thm:main}, one can define more general ``self-similar'' type
  doubling measures on $X$. Suppose for instance, that our space $X$ is such that the number
  of descendants of each cube $Q_{k,i}$ is at least $n\in\mathbb{N}$.
  Let $p_1,\ldots,p_n>0$ be positive numbers with
  $\sum_{i=1}^np_i=1$ and fix $0<p<M^{-2}$. Instead of \eqref{nudef}, we distribute the measure of $Q_{k,i}$
  among the descendants in the following way: For each $1 \le m \le n$ choose $j_m \in N_{k+1}$
  so that $Q_{k+1,j_m}\subset Q_{k,i}$ and $j_m \ne j_l$ when $m \ne l$. Define $\mu(Q_{k+2,j})=p\mu(Q_{k,i})$
  if $Q_{k+2,j}\subset Q_{k,i}$ and
  $x_{k+2,j}\notin\{x_{k+1,j_1},\ldots x_{k+1,j_n}\}$. Then divide the
  rest of the measure of $\mu(Q_{k,i})$ among the ``central subcubes''
  of $Q_{k,j_1},\ldots Q_{k,j_n}$ according to
  the probabilities $p_1,\ldots,p_n$.
\end{remark}

\begin{ack}
  We thank the referee for useful comments.
\end{ack}


\def\cprime{$'$} \def\cprime{$'$} \def\cprime{$'$}

\end{document}